\documentclass[12pt]{amsart}
\usepackage{amsmath}
\usepackage{amssymb}
\usepackage{amsthm}
\usepackage{amsfonts}
\usepackage{color}
\usepackage{hyperref}
\usepackage[margin=1.1in]{geometry}

\newtheorem{theorem}{Theorem}[section]
\newtheorem{lemma}[theorem]{Lemma}
\newtheorem{corollary}[theorem]{Corollary}
\newtheorem{prop}[theorem]{Proposition}
\newtheorem{question}[theorem]{Question}

\theoremstyle{definition}
\newtheorem{definition}[theorem]{Definition}

\theoremstyle{remark}
\newtheorem{remark}[theorem]{Remark}

\numberwithin{equation}{section}

\newcommand{\db}{\bar{\partial}}
\newcommand{\dd}{\partial}
\newcommand{\nablaB}{\nabla^{\text{\rm B}}}

\newcommand{\Ric}{\text{\rm Ric}}
\newcommand{\tr}{\text{\rm tr}}
\newcommand{\di}{\text{\rm div}}
\newcommand{\R}{\text{\rm Rm}}
\newcommand{\n}{\nabla}
\newcommand{\nb}{\overline{\nabla}}
\newcommand{\D}{\Delta}
\newcommand{\Db}{\overline{\Delta}}

\title{Bismut Einstein Metrics on Compact Complex Manifolds}
\author{Yanan Ye}
\email{\href{mailto:yeyanan@outlook.com}{yeyanan@outlook.com}}
\date{}

\begin{document}
\maketitle

\begin{abstract}
We observe that, for a Bismut Einstein metric, the (2,0)-part of Bismut Ricci form is an eigenvector of the Chern Laplacian.
With the help of this observation, we prove that a Bismut Einstein metric with non-zero Einstein constant is K{\"a}hler Einstein.
Additionally, for Bismut Einstein metrics with zero Einstein constant, we prove that they are actually Bismut Ricci flat.
\end{abstract}

\tableofcontents

\section{Introduction}
The pluriclosed flow introduced by Streets and Tian\cite{S-Tpcf,S-Thcf} is a parabolic flow, which evolves a pluriclosed metric by its Bismut Ricci curvature.
As Ricci flow and K{\"a}hler-Ricci flow, a Einstein-type metric (if exists) can generate a solution to the flow by a rescaling depended only on time.
This is the simplest case of solitons (see e.g., \cite{M-Tperelman,Perelman1,T-Zkrs}).
So understanding the property of Bismut Einstein metrics is helpful to investigate the behavior of pluriclosed flow, which can be applied to classify some manifolds, especially the Kodaira\rq{}s class VII surfaces (see \cite{MR4181011,MR2755684}).

Firstly, we give the definition of Bismut Einstein metrics.
\begin{definition}
A pluriclosed metric $\omega$ is called Bismut Einstein with Einstein constant $\lambda$ if
\begin{align}\label{eq-defBE}
\rho^{1,1}(\omega)=\lambda\omega.
\end{align}
\end{definition}
Here $\rho^{1,1}$ is the $(1,1)$-part of Bismut Ricci form that will be defined explicitly in Section \ref{subsec-Bismut}.
Since $\rho^{1,1}$ is not elliptic in general, we add pluriclosed condition in the definition.
But notice that in the case of $\lambda\neq0$, a solution to \eqref{eq-defBE} is automatically pluriclosed for $\rho^{1,1}$ is always pluriclosed.
To see this, we rewrite the Bismut Ricci form as $\rho=\rho^{1,1}+\rho^{2,0}+\rho^{0,2}$ by bi-degree.
Since $\rho$ is real and closed, we obtain
\begin{align*}
\dd\rho^{2,0}=0
,\qquad
\dd\rho^{1,1}+\db\rho^{2,0}=0.
\end{align*}

\begin{definition}
We say a manifold admits a Hermitian-symplectic (HS) structure if it admits a HS form, which is a real closed 2-form with positive definite (1,1)-part.
\end{definition}
K{\"a}hler metrics are special examples of HS forms.
And the (1,1)-part of any HS form is precisely pluriclosed.
In fact, a Bismut Einstein metric with $\lambda\neq0$ can be extended to a HS form as the (1,1)-part (see \cite{Ypcfhs}).

One motivation of this paper is a question asked by Streets and Tian in \cite{S-Tpcf}.
Using the classification of compact surfaces (see e.g., \cite{BHPA}), they show that only K{\"a}hler surfaces admit HS structures in dimension 2.
And they ask is it valid in high dimensional cases?
\begin{question}[Streets \& Tian in \cite{S-Tpcf}]\label{ques-HS}
Is there a non-K{\"a}hler manifold admitting HS structures?
\end{question}
Authors of \cite{Ypcfhs} find that HS forms are preserved by pluriclosed flow.
They deform a HS form along pluriclosed flow and prove that the limitation (if exists) must be K{\"a}hler in dimension 2.
This gives a way to consider Question \ref{ques-HS}.
One can study the canonical HS forms obtained by deforming HS forms along pluriclosed flow.
In this viewpoint, Bismut Einstein metric with $\lambda\neq0$ is an important class of canonical HS forms since it is the fixed point of pluriclosed flow up to a rescaling.
Then a natural question is that whether Bismut Einstein metrics with $\lambda\neq0$ are K{\"a}hler?
A classification of solitons given by Streets in \cite{Sscs} gives an affirmative answer in dimension 2. 
For high dimensional cases, this paper also gives an affirmative answer.

More precisely, we prove that

\begin{theorem}\label{the-introMain}
On compact complex manifolds, we have

(a) Bismut Einstein metrics with non-zero Einstein constant are K{\"a}hler Einstein;

(b) Bismut Einstein metrics with zero Einstein constant are Bismut Ricci flat.
\end{theorem} 
In the case of $\lambda=0$, there are some non-K{\"a}hler examples (see Example 2.7-Example 2.10 in \cite{GF-J-SLap}).
And all those examples are actually Bismut flat.

To prove Theorem \ref{the-introMain}, we establish a Bochner formula expressed in terms of Bismut Ricci curvature.
Meanwhile, we obtain some vanishing results under certain conditions on Bismut Ricci curvature and Bismut scalar curvature (see Section \ref{sec-Bochner}). 
More analogues vanishing theorems about Bismut connection can be found in \cite{Alexandrov-Ivanov,Ivanov-P}.
%
%


Here is an outline of the rest paper.
In section \ref{sec-Pre}, we recall some basic notions that will be used later.
In section \ref{sec-compLap}, we collect some calculation results about Chern Laplacian and $\db$-Laplacian, which are different in the non-K{\"a}hler case.
In section \ref{sec-Bochner}, we establish a Bochner formula expressed in terms of Bismut Ricci curvature and obtain some vanishing results of Dolbeault cohomology.
In section \ref{sec-BEgen}, we give an observation about Bismut Einstein metrics and use it to complete the proof of Theorem \ref{the-introMain}.
In section \ref{sec-proofMain}, we obtain a stronger vanishing theorem on surfaces.
As an application, we give another proof of Theorem \ref{the-introMain} in the surface case.
\newline

\textbf{Acknowledgments.} I want to express my gratitude to my advisor, Professor Gang Tian, for his helpful suggestions and patient guidance.
I also thanks Professor Jeffery Streets and Professor Mario Garcia-Fernandez for helpful comments and notifying me the paper \cite{GF-J-SLap}.
And thanks Professor Stefan Ivanov and Giuseppe Barbaro for their helpful comments on an earlier version.

\section{Preliminary}\label{sec-Pre}
In this section, we give a quick review of some basic notions which will be used later.

\subsection{Bismut connection and Bismut Ricci form}\label{subsec-Bismut}
Given a complex manifold $(M^{2n},J)$ with a Hermitian metric $g$.
Bismut connection $\nablaB$ is the unique connection satisfying
\begin{align*}
\nablaB g=0,\qquad \nablaB J=0, \qquad B(x,y,z)+B(z,y,x)=0,
\end{align*}
in which
\begin{align*}
B(x,y,z)=g(\nablaB_{x}y-\nablaB_{y}x-[x,y],z)
\end{align*}
is the tensor induced by torsion operator (see e.g., \cite{B-bc}).
Notice that $B$ is a real $3$-form, which is closed if and only if the metric is pluriclosed.

We denote $\rho$ the Ricci form of Bismut connection.
It is well know that $\rho$ is a closed real $2$ form (see e.g., \cite{C-cc}).
If we rewrite is as $\rho=\rho^{1,1}+\rho^{2,0}+\overline{\rho^{2,0}}$ by bidegree, then we have
\begin{align*}
\rho^{1,1}(\omega)
&=-\dd\dd^*\omega-\db\db^*\omega-\sqrt{-1}\dd\db\log\det g
\\
\rho^{2,0}(\omega)
&=-\dd\db^*\omega
\end{align*}
where $\omega(\cdot,\cdot)=g(J\cdot,\cdot)$ is the fundamental form.

\subsection{Chern connection without K{\"a}hler assumption}
\label{sec-Chern}
In this subsection we review some basic facts of Chern connection $\nabla$ without K{\"a}hler assumption.
In such a case, it is not Levi-Civita connection anymore.
And the torsion tensor is
\begin{align*}
T(x,y,z)=g(\nabla_x y-\nabla_y x-[x,y],z).
\end{align*}
In local coordinates, the Christoffel symbol of Chern connection is
\begin{align*}
\Gamma_{ij}^{s}=g^{\bar{t}s}\dd_{i}g_{j\bar{t}}
\end{align*}
Then
\begin{align*}
T_{ij\bar{t}}=	T(\dd_i,\dd_j,\dd_{\bar{t}})
=T_{ij}^{s}g_{s\bar{t}}
=\dd_{i}g_{j\bar{t}}-\dd_{j}g_{i\bar{t}}.
\end{align*}
Notice that
\begin{align*}
\dd\omega=\sqrt{-1}\dd_{i}g_{j\bar{t}}
					dz^i\wedge dz^j\wedge d\bar{z}^t
=\frac{\sqrt{-1}}{2}T_{ij\bar{t}}
					dz^i\wedge dz^j\wedge d\bar{z}^t
\end{align*}
So if we regard $\dd\omega$ as just a tensor, then we have 
$\dd\omega=\sqrt{-1}T$.

\begin{definition}\label{def-SandQ}
We can define the second Chern Ricci $S$ and a quadratic term $Q$ respectively by
\begin{align*}
S_{i\bar{j}}=g^{\bar{q}p}\Omega_{p\bar{q}i\bar{j}}
,\quad
Q_{i\bar{j}}=g^{\bar{q}p}g^{\bar{t}s}T_{is\bar{q}}\overline{T_{jt\bar{p}}}
=g^{\bar{q}p}g^{\bar{t}s}T_{is\bar{q}}T_{\bar{j}\bar{t}p}
\end{align*}
where $\Omega_{i\bar{j}p\bar{q}}$ is the Chern curvature tensor.
\end{definition}

Recall that the Chern Ricci is defined by 
$
\Ric_{i\bar{j}}=g^{\bar{q}p}\Omega_{i\bar{j}p\bar{q}}
$.
In K{\"a}hler case, $S$ is precisely the Chern Ricci.
In general, they are different since Chern connection has torsion.
An basic fact is that $S$ is alway a second order elliptic operator with respect to metrics.

For pluriclosed metrics, there is a relationship between $S$ and the $(1,1)$-part of Bismut Ricci form $\rho^{1,1}$ (see e.g., \cite{S-Tpcf,S-Thcf}). 
More explicitly, if we assume
\begin{align*}
\rho^{1,1}=\sqrt{-1}P_{i\bar{j}}dz^{i}\wedge d\bar{z}^{j}
\end{align*}
then we have
\begin{align}\label{eq-P11}
P=S-Q.
\end{align}

\begin{remark}\label{re-HnonNeg}
In fact, the quadratic term $Q$ is non-negative. 
To see this, we can choose a normal coordinates such that $g_{i\bar{j}}(x)=\delta_{ij}$ and $Q(x)$ is diagonal at a fixed point $x$.
Then we have
\begin{align*}
Q_{i\bar{i}}(x)=\sum_{s,p}T_{is\bar{p}}T_{\bar{i}\bar{s}p}
=\sum_{s,p}|T_{is\bar{p}}|^{2}\geq0
\end{align*}
\end{remark}

\subsection{Lefschetz-type operator}
\label{sec-Lfs}
For the convenience of use later, we recall the Lefschetz-type operator in this subsection.

\begin{definition}
Let $\gamma$ be a form.
The Lefschetz-type operator $L_{\gamma}$ is defined by
\begin{align*}
L_{\gamma}\alpha=\gamma\wedge\alpha.
\end{align*}
And its conjugate adjoint $L^{*}_{\gamma}$ is defined by
\begin{align*}
(L^{*}_{\gamma}\alpha,\beta)=(\alpha,\bar{\gamma}\wedge \beta).
\end{align*}
where $(\cdot,\cdot)=g(\cdot,\overline{\cdot})$ is the pointwise Hermitian inner product.
\end{definition}

Notice that in the case of $\gamma=\omega$, those operators defined above are the classical Lefschetz operator $L$ and $\Lambda$.

We give a local expression of the Lefschetz-type operator in a special case which will be used later.
\begin{lemma}\label{lem-Lfs}
Given a $(p,1)$-form
\begin{align*}
\alpha=\frac{1}{p!}\alpha_{i_1\cdots i_p \bar{k}}
			dz^{i_1}\wedge\cdots\wedge dz^{i_p}
			\wedge d\bar{z}^k
,\quad
\alpha_{\cdots i_u\cdots i_v\cdots \bar{k}}+\alpha_{\cdots i_v\cdots i_u\cdots \bar{k}}=0
\end{align*}
and a $(1,0)$-form
$\gamma=\gamma_{s}dz^{s}$.
We have
\begin{align*}
L^{*}_{\gamma}\alpha=\frac{1}{p!}(-1)^p g^{\bar{k}l}
			\alpha_{i_1\cdots i_p\bar{k}}\gamma_{l}
			dz^{i_1}\wedge\cdots\wedge dz^{i_p}.
\end{align*}
\end{lemma}
\begin{proof}
Choose an arbitrary $(p,0)$-form 
$\beta=\beta_{j_1\cdots j_p} dz^{j_1}\wedge\cdots\wedge dz^{j_p}$.
Assume
\begin{align*}
L^{*}_{\gamma}\alpha=\frac{1}{p!}
			\eta_{i_1\cdots i_p}
			dz^{i_1}\wedge\cdots\wedge dz^{i_p}.
\end{align*}
By direct computation, we obtain
\begin{align*}
(L^{*}_{\gamma}\alpha,\beta)
=g^{\bar{j_1}i_1}\cdots g^{\bar{j_p}i_p}
\eta_{i_1\cdots i_p}
\overline{\beta_{j_1\cdots j_p}}
\end{align*}
and
\begin{align*}
\bar{\gamma}\wedge\beta
=(-1)^p \beta_{j_1\cdots j_p}\overline{\gamma_{l}}
			dz^{j_1}\wedge\cdots\wedge dz^{j_p}
			\wedge d\bar{z}^l
\end{align*}
By definition,
\begin{align*}
(L^{*}_{\gamma}\alpha,\beta)
&=(\alpha,\bar{\gamma}\wedge\beta)
=(-1)^p g^{\bar{j_1}i_1}\cdots g^{\bar{j_p}i_p}g^{\bar{k}l}
					\alpha_{i_1\cdots i_p\bar{k}}\gamma_l
					\overline{\beta_{j_1\cdots j_p}}
\end{align*}
Then we get
\begin{align*}
g^{\bar{j_1}i_1}\cdots g^{\bar{j_p}i_p}
\eta_{i_1\cdots i_p}
\overline{\beta_{j_1\cdots j_p}}
=(-1)^p g^{\bar{j_1}i_1}\cdots g^{\bar{j_p}i_p}g^{\bar{k}l}
					\alpha_{i_1\cdots i_p\bar{k}}\gamma_l
					\overline{\beta_{j_1\cdots j_p}}
\end{align*}
which implies
\begin{align*}
\eta_{i_1\cdots i_p}
=(-1)^p g^{\bar{k}l}\alpha_{i_1\cdots i_p\bar{k}}\gamma_l
\end{align*}
\end{proof}	

\section{Some Calculation Results of Laplacians}\label{sec-compLap}
In this section, we collect some calculation results which will be used later.
For easy of notations, let us start with some definitions.

\begin{definition}
For a tensor $A$, we denote its gradient by
\begin{align*}
\n A=\n_{i}A_{\cdots}dz^{i}\otimes\cdots
,\quad
\nb A=\n_{\bar{i}}A_{\cdots}d\bar{z}^{i}\otimes\cdots
\end{align*}
We denote the Chern Laplacian by
\begin{align*}
\D=g^{\bar{q}p}\n_{p}\n_{\bar{q}}
,\quad
\Db=g^{\bar{q}p}\n_{\bar{q}}\n_{p}.
\end{align*}
We denote the $\dd$-Laplacian and $\db$-Laplacian by
\begin{align*}
\D_{\db}=\db^*\db+\db\db^*
,\qquad
\D_{\dd}=\dd^*\dd+\dd\dd^*.
\end{align*}
\end{definition}

\begin{definition}\label{def-circle}
Given two tensors $A_{\cdots i\bar{j}}$ and 
$B_{p_{1}\cdots p_{s}\bar{q}_{1}\cdots \bar{q}_{t}}$.
We can define a tensor $A\circ B$ by
\begin{align*}
(A\circ B)_{\cdots p_{1}\cdots p_{s}\bar{q}_{1}\cdots\bar{q}_{t} }
=&-\sum_{m=1}^{s}g^{\bar{\beta}\alpha}
	A_{\cdots p_{m}\bar{\beta}}B_{\cdots p_{m-1}\alpha p_{m+1}\cdots}
	\\
	&+\sum_{n=1}^{t}g^{\bar{\beta}\alpha}
	A_{\cdots \alpha \bar{q}_{n}}B_{\cdots \bar{q}_{n-1}\bar{\beta} \bar{q}_{n+1}\cdots}
\end{align*}
\end{definition}

\begin{remark}\label{re-secondOrderChern}
Using the notation defined above, we can rewrite the Chern curvature operator as
\begin{align*}
\n^{2}A(x,y)-\n^{2}A(y,x)
	=\Omega_{xy}\circ A=(\Omega\circ A)(x,y)
\end{align*}
Taking trace, we obtain
\begin{align*}
\D A-\Db A=S\circ A
\end{align*}
where $S$ is the second Chern Ricci defined in Section \ref{sec-Chern}.
\end{remark}

\begin{remark}\label{re-compS}
We would like to give a remark on the computation of $A\circ B$ in a special case where $B=B_{i_{1}\cdots i_{p}}\in (T^{*1,0}M)^{\otimes p}$ and $A=A_{i\bar{j}}$ is Hermitian, i.e., $\overline{A_{i\bar{j}}}=A_{j\bar{i}}$.
For convenience, we choose a normal coordinates for a fixed point $x$ such that
$g_{i\bar{j}}(x)=\delta_{ij}$ and
$A_{i\bar{j}}(x)=\lambda_{i}\delta_{ij}$.
Then by definition
\begin{align*}
(A\circ B)_{i_1\cdots i_p}(x)
=-(\sum_{m=1}^{p}\lambda_{i_m})B_{i_1\cdots i_p}.
\end{align*}
In particular, we have
\begin{align*}
g\circ B=-pB
\end{align*}
Moreover, if the subscripts of $B$ is skew-symmetric, then in the case of $p=n$ we have
\begin{align*}
(A\circ B)_{i_{1}\cdots i_{n}}(x)
=-(\sum_{m=1}^{n}\lambda_{m})B_{i_{1}\cdots i_{n}}
=-(\tr_{g}A)(x)B_{i_{1}\cdots i_{n}}
\end{align*}
Since $x$ is arbitrary, we get
\begin{align*}
A\circ B=-(\tr_{g}A)B.
\end{align*}
\end{remark}

We recall the formula of order exchange of Laplacian and gradient and leave the proof in the appendix.
\begin{lemma}\label{lem-exOf3derivative}
Let $A$ be a tensor field. 
We have
\begin{align*}
\n^{3}_{k\bar{k}\bar{l}}A-\n^{3}_{\bar{l}k\bar{k}}A
=\Omega_{k\bar{l}}\circ\n_{\bar{k}}A
	-\n_{k}T_{\bar{k}\bar{l}}^{\bar{s}}\n_{\bar{s}}A
	-T_{\bar{k}\bar{l}}^{\bar{s}}\n_{k}\n_{\bar{s}}A
\end{align*}
\end{lemma}

Then we give two lemmas often used in the computation of $\db$-Laplacian and $\dd$-Laplacian.
And we leave the tedious calculations in the appendix.

\begin{lemma}\label{lem-compDb}
Given a $(p,q+1)$-form
\begin{align*}
\alpha=\frac{1}{p!(q+1)!}
				\alpha_{i_1\cdots i_p \bar{j}_1\cdots \bar{j}_q \bar{k}}
				dz^{i_1}\wedge\cdots\wedge dz^{i_p}
				\wedge
				d\bar{z}^{j_1}\wedge\cdots\wedge d\bar{z}^{j_q}
				\wedge
				d\bar{z}^k
\end{align*}
satisfying
\begin{align*}
\alpha_{\cdots i_u \cdots i_v \cdots}
+\alpha_{\cdots i_v \cdots i_u \cdots}=0
,\quad
\alpha_{\cdots \bar{j}_u \cdots \bar{j}_v \cdots}
+\alpha_{\cdots \bar{j}_v \cdots \bar{j}_u \cdots}=0.
\end{align*}
We have
\begin{align*}
\db^*\alpha=\frac{1}{p!q!}
				\eta_{i_1\cdots i_p \bar{j}_1\cdots \bar{j}_q}
				dz^{i_1}\wedge\cdots\wedge dz^{i_p}
				\wedge
				d\bar{z}^{j_1}\wedge\cdots\wedge d\bar{z}^{j_q}
\end{align*}
in which
\begin{align*}
\eta_{i_1\cdots i_p \bar{j}_1\cdots \bar{j}_q}
=(-1)^{p+q+1}g^{\bar{k}l}
			\Big(
			\nabla_{l}\alpha_{i_1\cdots i_p \bar{j}_1\cdots \bar{j}_q \bar{k}}
+\alpha_{i_1\cdots i_p \bar{j}_1\cdots \bar{j}_q \bar{k}}
			T_{ls}^{s}
+\frac{1}{2}g^{\bar{t}s}
			\sum_{m=1}^{q}
			\alpha_{i_1\cdots i_p \bar{j}_1 \cdots\bar{t}\cdots\bar{j}_q \bar{k}}
			T_{sl\bar{j}_m}
			\Big).
\end{align*}
\end{lemma}

Similarly, we have
\begin{lemma}\label{lem-compDd}
Given a $(p+1,q)$-form
\begin{align*}
\alpha=\frac{1}{(p+1)!q!}
				\alpha_{l i_1\cdots i_p \bar{j}_1\cdots \bar{j}_q}
				dz^{l}\wedge
				dz^{i_1}\wedge\cdots\wedge dz^{i_p}
				\wedge
				d\bar{z}^{j_1}\wedge\cdots\wedge d\bar{z}^{j_q}
\end{align*}
satisfying
\begin{align*}
\alpha_{\cdots i_u \cdots i_v \cdots}
+\alpha_{\cdots i_v \cdots i_u \cdots}=0
,\quad
\alpha_{\cdots \bar{j}_u \cdots \bar{j}_v \cdots}
+\alpha_{\cdots \bar{j}_v \cdots \bar{j}_u \cdots}=0.
\end{align*}
We have
\begin{align*}
\dd^*\alpha=\frac{1}{p!q!}
				\eta_{i_1\cdots i_p \bar{j}_1\cdots \bar{j}_q}
				dz^{i_1}\wedge\cdots\wedge dz^{i_p}
				\wedge
				d\bar{z}^{j_1}\wedge\cdots\wedge d\bar{z}^{j_q}
\end{align*}
in which
\begin{align*}
\eta_{i_1\cdots i_p \bar{j}_1\cdots \bar{j}_q}
=-g^{\bar{k}l}
			\Big(
			\nabla_{\bar{k}}
			\alpha_{l i_1\cdots i_p \bar{j}_1\cdots \bar{j}_q}
+\alpha_{l i_1\cdots i_p \bar{j}_1\cdots \bar{j}_q}
			\overline{T_{ks}^{s}}
-\frac{1}{2}g^{\bar{t}s}
			\sum_{m=1}^{q}
			\alpha_{l i_1\cdots s \cdots i_p \bar{j}_1 \cdots\bar{j}_q}
			\overline{T_{kt\bar{i}_m}}
			\Big).
\end{align*}
\end{lemma}

\begin{remark}
In Lemma \ref{lem-compDb}, when $q=0$, we do not have the last term in the expression of 
$\eta_{i_1\cdots i_p \bar{j}_1\cdots \bar{j}_q}$.
In other words, the expression for $(p,1)$-form is
\begin{align*}
\eta_{i_1\cdots i_p}
=(-1)^{p+1}g^{\bar{k}l}
			\Big(
			\nabla_{l}\alpha_{i_1\cdots i_p \bar{k}}
+\alpha_{i_1\cdots i_p \bar{k}}
			T_{ls}^{s}
			\Big).
\end{align*}
This is also valid for Lemma \ref{lem-compDd} in the case of $p=0$.
\end{remark}

\begin{remark}\label{rm-dbOmega}
Applying Lemma \ref{lem-compDb} to $\omega$, we have
\begin{align*}
\db^*\omega=\sqrt{-1}T_{is}^{s}dz^i.
\end{align*}
\end{remark}

\subsection{Laplacians applied to functions}
Next proposition gives the relationship between two Laplacians when applied to functions.

\begin{prop}\label{prop-LapOnf}
For a smooth function $f$, we have
\begin{align*}
-\Delta_{\db}f=\Delta f -\sqrt{-1}(\db f,\dd^*\omega)\\
-\Delta_{\dd}f=\Delta f +\sqrt{-1}(\dd f,\db^*\omega)
\end{align*}
\end{prop}
\begin{proof}
For Chern Laplacian, we have
\begin{align*}
\Delta f=\tr_{\omega}(\sqrt{-1}\dd\db f)
=\tr_{\omega}(\sqrt{-1}\dd_{i\bar{j}}f dz^i\wedge d\bar{z}^j)
=g^{\bar{q}p}\dd_{p\bar{q}}f
\end{align*}
For $\db$-Laplacian, we have
\begin{align*}
\Delta_{\db}f=\db\db^*f+\db^*\db f
=\db^*\db f
\end{align*}
Applying Lemma \ref{lem-compDb} to $\db f$ and noticing Remark \ref{rm-dbOmega}, we get
\begin{align*}
\Delta_{\db}f
&=-g^{\bar{k}l}\nabla_{l}(\db f)_{\bar{k}}
			-g^{\bar{k}l}(\db f)_{\bar{k}}T_{ls}^{s}\\
&=-g^{\bar{k}l}\dd_{l\bar{k}} f
			+\sqrt{-1}g^{\bar{k}l}(\db f)_{\bar{k}}(\db^*\omega)_l\\
&=-\Delta f +\sqrt{-1}(\db f,\dd^*\omega)
\end{align*}
Similarly, applying Lemma \ref{lem-compDd}, we get
\begin{align*}
\Delta_{\dd}f
&=-g^{\bar{k}l}\nabla_{\bar{k}}(\dd f)_{l}
			-g^{\bar{k}l}(\dd f)_{l}\overline{T_{ks}^{s}}\\
&=-g^{\bar{k}l}\dd_{l\bar{k}} f
			-\sqrt{-1}g^{\bar{k}l}(\dd f)_{l}(\dd^*\omega)_{\bar{k}}\\
&=-\Delta f -\sqrt{-1}(\dd f,\db^*\omega)
\end{align*}
\end{proof}

\begin{remark}
From Proposition \ref{prop-LapOnf}, we know that if $\db^*\omega=0$, then
$-\Delta_{\db}f=-\Delta_{\dd}f=\Delta f$ for all functions.
Notice that
\begin{align*}
\db^*\omega=0 \iff -*\dd*\omega=0 \iff \dd\omega^{n-1}=0
\end{align*}
So Chern Laplacian and $\db$-Laplacian are the same when applied to functions if and only if the metric is balanced, i.e., $d\omega^{n-1}=0$,
\end{remark}

Recall that
\begin{align*}
\int_{M} \Delta_{\db}f dV =(\Delta_{\db}f,1)_2
=(\db^*\db f,1)_2=(\db f,\db 1)_2=0 
\end{align*}
where $(\cdot,\cdot)_2$ is the $L^{2}$ Hermitian inner product.
But for Chern Laplacian, this is not valid in general.
Actually, we have 
\begin{prop}\label{prop-GauduchonIntLap=0}
If and only if the metric is Gauduchon (i.e., $\dd\db\omega^{n-1}=0$), we have
\begin{align*}
\int_{M}\Delta fdV=0
\end{align*}
for all functions.
\end{prop}
\begin{proof}
By definition, we have
\begin{align*}
\int_{M}\Delta fdV
&=\int_{M}\tr_{\omega}(\sqrt{-1}\dd\db f)dV
=\int_{M}(\sqrt{-1}\dd\db f,\omega)dV
\\
&=(\sqrt{-1}\dd\db f,\omega)_2=(\sqrt{-1}f,\db^*\dd^*\omega)_2
\end{align*}
Notice
\begin{align*}
\db^*\dd^*\omega=0 \iff *\dd\db*\omega=0 \iff \dd\db\omega^{n-1}=0.
\end{align*}
So we complete our proof.
\end{proof}

\begin{remark}
Gauduchon\cite{Gfrench} proves that any complex manifold admits metrics satisfying $\dd\db\omega^{n-1}=0$.
And in the case of surface, Gauduchon condition is precisely pluriclosed condition.
\end{remark}

\subsection{Laplacians applied to $(p,0)$-forms}
For the purpose of use later, we give the relationship between those Laplacians when applied on $(p,0)$-forms.

\begin{prop}\label{prop-LapForm}
For a $(p,0)$-form $\alpha$, we have
\begin{align*}
\Delta_{\db}\alpha=-\Delta\alpha
			+(-1)^p\sqrt{-1}L^{*}_{\db^*\omega}\db\alpha
\end{align*}
\end{prop}
where $L^{*}_{(\cdot)}(\cdot)$ is the Lefschetz-type operator defined in section \ref{sec-Lfs}.

\begin{proof}
In local coordinates, we assume
\begin{align*}
\alpha=\frac{1}{p!}\alpha_{i_1\cdots i_p}
				dz^{i_1}\wedge\cdots\wedge dz^{i_p}
,\qquad \alpha_{\cdots i_u\cdots i_v\cdots}
						+\alpha_{\cdots i_v\cdots i_u\cdots}=0.
\end{align*}
By definition,
\begin{align*}
\db\alpha=\frac{1}{p!}(-1)^p \dd_{\bar{k}}\alpha_{i_1\cdots i_p}
			dz^{i_1}\wedge\cdots\wedge dz^{i_p}
			\wedge d\bar{z}^k.
\end{align*}
Applying Lemma \ref{lem-compDb} to $\db\alpha$, we obtain
\begin{align*}
\db^*\db\alpha
&=-\frac{1}{p!}(g^{\bar{k}l}\nabla_{l}\dd_{\bar{k}}\alpha_{i_1\cdots i_p}
			+g^{\bar{k}l}\dd_{\bar{k}}\alpha_{i_1\cdots i_p}T_{ls}^{s})
			dz^{i_1}\wedge\cdots\wedge dz^{i_p}
\end{align*}
From Lemma \ref{lem-Lfs} and Remark \ref{rm-dbOmega}, we know
\begin{align*}
-\frac{1}{p!}g^{\bar{k}l}\dd_{\bar{k}}\alpha_{i_1\cdots i_p}T_{ls}^{s}
			dz^{i_1}\wedge\cdots\wedge dz^{i_p}
&=\sqrt{-1}\frac{1}{p!}g^{\bar{k}l}
			\dd_{\bar{k}}\alpha_{i_1\cdots i_p}(\db^*\omega)_{l}
			dz^{i_1}\wedge\cdots\wedge dz^{i_p}
\\
&=(-1)^p \sqrt{-1}L^{*}_{\db^*\omega}\db\alpha.
\end{align*}
This implies
\begin{align*}
\Delta_{\db}\alpha=-\Delta\alpha
			+(-1)^p\sqrt{-1}L^{*}_{\db^*\omega}\db\alpha.
\end{align*}
\end{proof}

\section{A Bochner Formula}\label{sec-Bochner}
In this section, we give a Bochner formula for pluriclosed metrics in terms of Bismut Ricci curvature.
As an application, we can obtain some vanishing results on the the Dolbeault cohomology $H_{\db}^{p,0}(M;\mathbb{C})$.

\subsection{Bochner formula}\label{subsec-BochnerFormula}
For ease of notations, we start with some definition.

\begin{definition}\label{def-1eigenvalue}
Given a tensor $A=A_{i\bar{j}}$ satisfying $\overline{A_{i\bar{j}}}=A_{j\bar{i}}$.
We define the first eigenvalue function $\lambda_{*}(A)$ by
\begin{align*}
\lambda_{*}(A)(x)=\min_{0\neq \xi\in T^{1,0}_{x}M}
			\frac{A(x)(\xi,\overline{\xi})}{|\xi|^2}.
\end{align*}
\end{definition}
\begin{remark}\label{re-lambda1}
We choose a normal coordinates at a fixed point $x$ such that
$g_{i\bar{j}}(x)=\delta_{ij}$
and
$A_{i\bar{j}}(x)=\lambda_i\delta_{ij}$.
It is easy to see that
$\lambda_{*}(A)(x)=\min\{\lambda_{1},\cdots,\lambda_{n}\}$.
Thus $A$ is non-negative (resp. positive) if and only if $\lambda_{*}(A)$ is a non-negative (resp. positive) function.
\end{remark}

Now we can state the Bochner formula.
\begin{theorem}\label{the-BochnerRn}
For any tensor $A$, we have
\begin{align*}
\Delta|A|^2=(\D A,A)+(A,\D A)+|\n A|^2+|\nb A|^2-(A,S\circ A)
\end{align*}
where $S$ is the second Chern Ricci and $(\cdot,\cdot)$ denotes the Hermitian inner product.
\end{theorem}

\begin{proof}
By definition and noticing that Chern connection is compatible with metric, we obtain
\begin{align*}
\D|A|^2
&=g^{\bar{q}p}\n_{p}\n_{\bar{q}}g(A,\overline{A})
\\
&=g^{\bar{q}p}\Big\{
	g(\n_{p}\n_{\bar{q}}A,\overline{A})
	+g(A,\n_{p}\n_{\bar{q}}\overline{A})
	+g(\n_{p}A,\n_{\bar{q}}\overline{A})
	+g(\n_{\bar{q}}A,\n_{p}\overline{A})
	\Big\}
\\
&=(\D A,A)+(A,\Db A)+|\n A|^2+|\nb A|^2
\end{align*}
Remark \ref{re-secondOrderChern} tells us that
\begin{align*}
\Db A=\D A-S\circ A,
\end{align*}
which completes this proof.
\end{proof}

From equation \eqref{eq-P11}, we can directly obtain
\begin{corollary}\label{cor-BochnerPLC}
If the metric is pluriclosed, then we have
\begin{align*}
\D|A|^2=(\D A,A)+(A,\D A)
			+|\n A|^2+|\nb A|^2
			-(A,P\circ A)
			-(A,Q\circ A)
\end{align*}
\end{corollary}

\subsection{Some vanishing results}\label{sec-vangeneral}
In this subsection, we apply the Bochner formula established in Section \ref{subsec-BochnerFormula} to get some vanishing results.
Firstly, we state a lemma that will be used repeatedly.

\begin{lemma}\label{lem-positiveQphi}
If $\alpha$ is a $(p,0)$-form and $A=A_{i\bar{j}}$ is Hermitian, then we have
\begin{align*}
-(\alpha,A\circ\alpha)\geq \lambda_{*}(A)|\alpha|^2
\end{align*}
\end{lemma}
\begin{proof}
We check it in a normal coordinates at a point $x$ such that $g_{i\bar{j}}(x)=\delta_{ij}$ and $A_{i\bar{j}}(x)=\lambda_{i}\delta_{ij}$.
From Remark \ref{re-compS} and Remark \ref{re-lambda1}, we know that
\begin{align*}
-(\alpha,A\circ\alpha)(x)
&=-\sum_{i_1,\cdots ,i_p}\alpha_{i_1\cdots i_p}
			\overline{(A\circ\alpha)_{i_1\cdots i_p}}
=\sum_{i_1,\cdots ,i_p}(\sum_{m=1}^{p}\lambda_{i_m})
			\alpha_{i_1\cdots i_p}
			\overline{\alpha_{i_1\cdots i_p}}
\\
&\geq\lambda_{*}(A)(x)\cdot|\alpha|^2(x).
\end{align*}
\end{proof}

\begin{prop}\label{prop-vanRn}
Given a compact complex manifold $(M^{2n},J)$. 
If $M$ admits a metric such that the second Chern Ricci $S$ is positive, then the Dolbeault cohomology $H_{\db}^{p,0}(M;\mathbb{C})$ is trivial for $1\leq p \leq n$.
\end{prop}
\begin{proof}
Given a  $\db$-harmonic $(p,0)$-form $\alpha$.
We have $\db\alpha=\db^*\alpha=0$.
From Proposition \ref{prop-LapForm}, we obtain	
\begin{align*}
\Delta\alpha=-\Delta_{\db}\alpha
			+(-1)^p\sqrt{-1}L^{*}_{\db^*\omega}\db\alpha
=0.
\end{align*}
Applying Theorem \ref{the-BochnerRn} to $\alpha$, we get
\begin{align}\label{eq-vanishingEq}
\Delta|\alpha|^2=|\nabla\alpha|^2+|\overline{\nabla}\alpha|^2
			-(\alpha,S\circ\alpha).
\end{align}
By Lemma \ref{lem-positiveQphi}, we have
\begin{align*}
-(\alpha,S\circ\alpha)\geq\lambda_{*}(S)|\alpha|^2
\end{align*}

Assume $|\alpha|^2$ achieves the maximum at point $x_M$.
By the maximum principle, we have
\begin{align*}
0\geq\Delta|\alpha|^2(x_M)
&=|\nabla\alpha|^2(x_M)
			+|\overline{\nabla}\alpha|^2(x_M)
			-(\alpha,S\circ\alpha)(x_M)
\\
&\geq |\nabla\alpha|^2(x_M)
			+|\overline{\nabla}\alpha|^2(x_M)
			+\lambda_{*}(S)(x_M)\cdot|\alpha|^2(x_M)
\\
&\geq 0
\end{align*}
The last equality uses the positive assumption on $S$.
So we obtain $|\alpha|^2(x_M)=0$, which implies $\alpha=0$.
Since $H_{\db}^{p,0}(M;\mathbb{C})\simeq \ker\Delta_{\db}\Big|_{\Lambda^{p,0}}$, we complete this proof.
\end{proof}

In particular, for $H_{\db}^{n,0}(M;\mathbb{C})$, we only need the assumption that Chern scalar curvature $s=\tr_{\omega}S$ is positive.
\begin{prop}\label{prop-vanScalar}
Given a compact complex manifold $(M^{2n},J)$.
If $M$ admits a metric such that the Chern scalar curvature $s=\tr_{\omega}S$ is positive, then the Dolbeault cohomology $H_{\db}^{n,0}(M;\mathbb{C})$ is trivial.
\end{prop}
\begin{proof}
Remark \ref{re-compS} tells that
\begin{align*}
-(\alpha,S\circ\alpha)=(\alpha,\tr_{g}S\cdot\alpha)
=s|\alpha|^2
\end{align*}
for any $(n,0)$-form $\alpha$.
And the rest argument is similar to Proposition \ref{prop-vanRn}.
\end{proof}

For Gauduchon metric, we can weaken the assumption slightly.

\begin{prop}\label{prop-vanGauduchonRn}
Given a compact complex manifold $(M^{2n},J)$.
If $M$ admits a Gauduchon metric $\omega$ such that $S$ is non-negative and is strictly positive at one point, then the Dolbeault cohomology $H_{\db}^{p,0}(M;\mathbb{C})$ is trivial for $1\leq p \leq n$.
\end{prop}
\begin{proof}
Let $\alpha$ be a $\db$-harmonic $(p,0)$-form.
Integrating equation \eqref{eq-vanishingEq} on $M$ and applying Proposition \ref{prop-GauduchonIntLap=0}, we obtain
\begin{align*}
0=\Vert\nabla\alpha\Vert^2+\Vert\overline{\nabla}\alpha\Vert^2
			+\int_{M}-(\alpha,S\circ\alpha)dV
\end{align*}
where $\Vert\cdot\Vert$ denotes the $L^2$ norm.
Since $-(\alpha,S\circ\alpha)$ is a non-negative function, we get
\begin{align*}
\Vert\nabla\alpha\Vert^2=\Vert\overline{\nabla}\alpha\Vert^2
=\int_{M}-(\alpha,S\circ\alpha)dV=0,
\end{align*}
which means $\nabla\alpha=\overline{\nabla}\alpha=0$.
We claim that $|\alpha|^2$ is a constant function on $M$.
To see this, we take derivatives to $|\alpha|^2$.
\begin{align*}
\dd_i |\alpha|^2=\nabla_{i} (\alpha,\alpha)
=(\nabla_{i}\alpha,\alpha)+(\alpha,\nabla_{\bar{i}}\alpha)=0
\end{align*}
Similarly, we have $\dd_{\bar{i}}|\alpha|^2=0$.

On the other hand, we have
\begin{align*}
0&=\int_{M}-(\alpha,S\circ\alpha)dV
\geq\int_{M}\lambda_{*}(S)|\alpha|^2dV
=|\alpha|^2\int_{M}\lambda_{*}(S)dV
\end{align*}
The integral on the right hand side is a positive number for $\lambda_{*}(S)$ is strictly positive at one point.
So we get $\alpha=0$ and complete this proof.
\end{proof}

Similarly, for $H^{n,0}_{\db}(M;\mathbb{C})$ we only need the assumption on the Chern scalar curvature $s$.
\begin{prop}\label{prop-nDolbeaultGauduchon}
Given a compact complex manifold $(M^{2n},J)$.
If $M$ admits a Gauduchon metric such that the Chern scalar curvature $s$ is non-negative and is strictly positive at one point, then the Dolbeault cohomology $H_{\db}^{n,0}(M;\mathbb{C})$ is trivial.
\end{prop}
\begin{proof}
The argument is similar to Proposition \ref{prop-vanGauduchonRn}.
Firstly, we prove that $|\alpha|^2$ is a constant function by integration.
Then use the equation
$
-(\alpha,S\circ\alpha)=s|\alpha|^2
$
to show that $|\alpha|^2$ is actually $0$.
\end{proof}

For pluriclosed metric, the second Chern Ricci can be represented by Bismut Ricci.
Thus we can state the vanishing result in terms of Bismut Ricci.
In other words, the positivity of the $(1,1)$-part of Bismut Ricci form can give obstructions to the Dolbeault cohomology $H^{p,0}_{\db}(M;\mathbb{C})$.
\begin{theorem}\label{thm-vanGeneral}
Given a compact complex manifold with a pluriclosed metric $(M^{2n},J,\omega)$.

(a) If $\rho^{1,1}$ is strictly positive definite, then $H_{\db}^{p,0}(M;\mathbb{C})$ is trivial for $1\leq p \leq n$;

(b) If Bismut scalar curvature $r=\tr_{\omega}\rho^{1,1}$ is positive, then $H_{\db}^{n,0}(M;\mathbb{C})$ is trivial.
\end{theorem}
\begin{proof}
Since equation \eqref{eq-P11} and $Q\geq0$, we can obtain (a) and (b) from Proposition \ref{prop-vanRn} and Proposition \ref{prop-vanScalar}, respectively.
\end{proof}

\section{Bismut Einstein Metrics on Compact Complex Manifolds}\label{sec-BEgen}

In this section, we will discuss the Bismut Einstein metric on compact complex manifolds and use the Bochner formula established in Section \ref{subsec-BochnerFormula} to give a proof of Theorem \ref{the-introMain}.
Firstly, we introduce an observation that plays an important role in the proof of the main theorem.

We assume in local coordinates
\begin{align*}
\rho^{2,0}=\frac{\sqrt{-1}}{2}\phi_{ij}dz^{i}\wedge dz^{j}
\end{align*}
where $\phi_{ij}+\phi_{ji}=0$.
For convenience, we will also use $\phi=-\sqrt{-1}\rho^{2,0}$ to denote the $(2,0)$-part of Bismut Ricci in the following.

\subsection{An observation of Bismut Einstein metrics}

Let us begin with a formula given by authors of \cite{GF-J-SLap}.

\begin{lemma}[\textit{Proposition 3.24 of \cite{GF-J-SLap}}]\label{lem-ChernDiv}
The $(2,0)$-part of Bismut Ricci can be represented by
\begin{align*}
\phi=-\di T,
\end{align*}
where $(\di T)_{ij}=-g^{\bar{k}l}\nabla_{l}T_{ij\bar{k}}$.
\end{lemma}

Now we give an observation about Bismut Einstein metrics.
\begin{prop}\label{prop-observationGeneral}
If $\omega$ is a Bismut Einstein metric with Einstein constant $\lambda$, then we have
\begin{align*}
\D\phi=-\lambda\phi.
\end{align*}
Moreover, in the case of $\lambda\neq0$ we have
\begin{align*}
\lambda T+\nb\phi=0.
\end{align*}
\end{prop}
\begin{proof}
In the case of $\lambda=0$.
Since $\rho=\rho^{2,0}+\overline{\rho^{2,0}}$ is closed, we obtain that $\rho^{2,0}$ is closed.
Applying Lemma \ref{prop-LapForm}, we get
\begin{align*}
\Delta \rho^{2,0}=-\Delta_{\db}\rho^{2,0}=0.
\end{align*}
Recall that $\phi=-\sqrt{-1}\rho^{2,0}$.
So we obtain $\D\phi=0$. 

From now on we assume $\lambda\neq0$.
We have
\begin{align*}
\dd\omega=\frac{1}{\lambda}\dd\rho^{1,1}
=-\frac{1}{\lambda}\db\rho^{2,0}
\end{align*}
In local coordinates, the equation above becomes
\begin{align*}
\frac{\sqrt{-1}}{2}T_{ij\bar{k}}dz^i\wedge dz^j\wedge d\bar{z}^k
=-\frac{1}{\lambda}\frac{\sqrt{-1}}{2}\dd_{\bar{k}}\phi_{ij}
					dz^i\wedge dz^j\wedge d\bar{z}^k
\end{align*}
which implies
\begin{align}\label{eq-HSTorsionphi}
T_{ij\bar{k}}
=-\frac{1}{\lambda}\dd_{\bar{k}}\phi_{ij}
=-\frac{1}{\lambda}\n_{\bar{k}}\phi_{ij}.
\end{align}
Applying Lemma \ref{lem-ChernDiv}, we obtain
\begin{align*}
\phi_{ij}=g^{\bar{k}l}\nabla_{l}T_{ij\bar{k}}
=-\frac{1}{\lambda}g^{\bar{k}l}\n_{l}\n_{\bar{k}}\phi_{ij}
=-\frac{1}{\lambda}\Delta \phi_{ij}
\end{align*}
This completes the proof.
\end{proof}

Next, we state a differential equation which will be used later.

\begin{lemma}\label{lem-LaplacianOf|phi|}
Assume $\omega$ is a Bismut Einstein metric with Einstein constant $\lambda$.
Then we have
\begin{align}\label{eq-LaplacianOf|phi|}
\D|\phi|^2=|\n\phi|^2+|\nb\phi|^2-(\phi,Q\circ\phi)
\end{align}
where $\phi=-\sqrt{-1}\rho^{2,0}$.
\end{lemma}
\begin{proof}
Applying Corollary \ref{cor-BochnerPLC} to $\phi$, we obtain
\begin{align*}
\Delta|\phi|^2=(\D \phi,\phi)+(\phi,\D \phi)
	+|\n \phi|^2+|\nb \phi|^2-(\phi,P\circ \phi)-(\phi,Q\circ \phi).
\end{align*}
From Remark \ref{re-compS} we know
\begin{align*}
-(\phi,P\circ\phi)=-(\phi,(\lambda g)\circ\phi)=2\lambda|\phi|^2
\end{align*}
Then applying Proposition \ref{prop-observationGeneral}, we get
\begin{align*}
\Delta|\phi|^2
&=(-\lambda\phi,\phi)+(\phi,-\lambda\phi)
	+|\n \phi|^2+|\nb \phi|^2+2\lambda|\phi|^2-(\phi,Q\circ \phi)
\\
&=|\n \phi|^2+|\nb \phi|^2-(\phi,Q\circ \phi)
\end{align*}
\end{proof}

\subsection{Proof of Theorem \ref{the-introMain}}
\begin{proof}[Proof of Theorem \ref{the-introMain}]
From Remark \ref{re-HnonNeg} and Lemma \ref{lem-positiveQphi}, we know that
\begin{align*}
-(\phi,Q\circ\phi)\geq 0.
\end{align*}
So Lemma \ref{lem-LaplacianOf|phi|} shows that
\begin{align*}
\D|\phi|^2\geq 0
\end{align*}
Notice that the manifold is compact.
Applying the strong maximum principle to $|\phi|^2$, we obtain that $|\phi|^2$ is a constant function.
Thus equation \eqref{eq-LaplacianOf|phi|} becomes
\begin{align*}
0=|\n\phi|^2+|\nb\phi|^2-(\phi,Q\circ\phi)
\end{align*}
Since the three terms on the right hand side are non-negative, we obtain
\begin{align*}
\n\phi=\nb\phi=0
,\quad
-(\phi,Q\circ\phi)=0.
\end{align*}

{\bf Case I $\lambda\neq0$.}
From Proposition \ref{prop-observationGeneral} we know that
\begin{align*}
T=-\frac{1}{\lambda}\nb\phi=0,
\end{align*}
This shows that the metric is K{\"a}hler and thus K{\"a}hler-Einstein.

{\bf Case II $\lambda=0$.}
By definition \ref{def-circle} and definition \ref{def-SandQ}, we have
\begin{align*}
-(\phi,Q\circ\phi)
&=g^{\bar{s}i}g^{\bar{t}j}\phi_{ij}
	\overline{
		g^{\bar{b}a}(Q_{s\bar{b}}\phi_{at}+Q_{t\bar{b}}\phi_{sa})
	}
=2g^{\bar{s}i}g^{\bar{t}j}g^{\bar{a}b}
	Q_{b\bar{s}}\phi_{ij}\phi_{\bar{a}\bar{t}}
\\
&=2g^{\bar{s}i}g^{\bar{t}j}g^{\bar{a}b}g^{\bar{q}p}g^{\bar{l}k}
	T_{bp\bar{l}}T_{\bar{s}\bar{q}k}\phi_{ij}\phi_{\bar{a}\bar{t}}
\\
&=2g^{\bar{t}j}g^{\bar{q}p}g^{\bar{l}k}
	(g^{\bar{a}b}T_{bp\bar{l}}\phi_{\bar{a}\bar{t}})
	(g^{\bar{s}i}T_{\bar{s}\bar{q}k}\phi_{ij})
\\
&=2g^{\bar{t}j}g^{\bar{q}p}g^{\bar{l}k}
	(g^{\bar{a}b}T_{bp\bar{l}}\phi_{\bar{a}\bar{t}})
	(\overline{g^{\bar{i}s}T_{sq\bar{k}}\phi_{\bar{i}\bar{j}}})
\end{align*}
Notice that the right hand side is the norm of the tensor
$A_{p\bar{l}\bar{t}}=g^{\bar{a}b}T_{bp\bar{l}}\phi_{\bar{a}\bar{t}}$.
Thus we obtain
\begin{align*}
g^{\bar{a}b}T_{bp\bar{l}}\phi_{\bar{a}\bar{t}}=0.
\end{align*}
Differentiating both sides and noticing $\nb\phi=\nb\phi=0$, we get
\begin{align*}
0=\n_{k}(g^{\bar{a}b}T_{bp\bar{l}}\phi_{\bar{a}\bar{t}})
=g^{\bar{a}b}\n_{k}T_{bp\bar{l}}\phi_{\bar{a}\bar{t}}
\end{align*}
Taking trace and applying Lemma \ref{lem-ChernDiv}, we get
\begin{align*}
g^{\bar{a}b}\phi_{bp}\phi_{\bar{a}\bar{t}}=0
\end{align*}
So we obtain $|\phi|^2=0$ by taking trace again.
The proof is completed since $\rho^{2,0}=\sqrt{-1}\phi$.
\end{proof}

\section{The Case of Surface}\label{sec-proofMain}

In section \ref{sec-vangeneral}, we show that for Gauduchon metrics, those vanishing results still hold under a slightly weaker conditions.
And in dimension 2, Gauduchon condition is precisely the pluriclosed condition.
So we can obtain a slightly stronger vanishing result in the case of surfaces.
As an application, we can give another proof of Theorem \ref{the-introMain} in the case of surfaces.

\subsection{A vanishing theorem for surfaces}

Let us begin with an observation in dimension 2.
\begin{lemma}\label{lem-Htwo2}
In dimension 2, we have
\begin{align*}
Q=\frac{1}{2}|T|^2 g
\end{align*}
\end{lemma}
\begin{proof}
We prove it in a normal coordinates such that $g_{i\bar{j}}(x)=\delta_{ij}$ and $Q_{i\bar{j}}(x)=\lambda_{i}\delta_{ij}$.
By definition, we have
\begin{align*}
Q_{1\bar{1}}(x)=\sum_{i,j}T_{1i\bar{j}}T_{\bar{1}\bar{i}j}
=\sum_{j}|T_{12\bar{j}}|^2
\end{align*}
Here we use the fact that the first two subscripts of $T$ are skew-symmetric.
Similarly, we have
\begin{align*}
Q_{2\bar{2}}(x)=\sum_{i,j}T_{2i\bar{j}}T_{\bar{2}\bar{i}j}
=\sum_{j}|T_{21\bar{j}}|^2
\end{align*}
Thus we complete this proof.
\end{proof}

Now we state a vanishing theorem for compact surfaces.
\begin{theorem}\label{thm-introVan2}
Given a compact complex surface with a pluriclosed metric $(M^{4},J,\omega)$.

(a) If $\rho^{1,1}$ is non-negative definite, then either $\omega$ is K{\"a}hler or $H_{\db}^{p,0}(M;\mathbb{C})$ is trivial for $1\leq p \leq 2$;

(b) If Bismut scalar curvature $r=\tr_{\omega}\rho^{1,1}$ is non-negative, then either $\omega$ is K{\"a}hler or $H_{\db}^{2,0}(M;\mathbb{C})$ is trivial.
\end{theorem}
\begin{proof}
Notice that in dimension 2, pluriclosed metrics are Gauduchon metrics.
And from Lemma \ref{lem-Htwo2} we know that
\begin{align*}
S=P+Q=P+\frac{1}{2}|T|^2 g
\end{align*}

If $\omega$ is non-K{\"a}hler, then there exists a point $x$ such that $|T|^2(x)>0$.
So part (a) and part (b) are obtained from Proposition \ref{prop-vanGauduchonRn} and Proposition \ref{prop-nDolbeaultGauduchon}, respectively.
\end{proof}

\subsection{Bismut Einstein metrics on surfaces}

We first give an proposition that is similar to Proposition \ref{prop-observationGeneral}.
\begin{prop}\label{prop-observation}
Given a Bismut Einstein metric with Einstein constant $\lambda$ on a complex surface.
We have
\begin{align*}
\Delta_{\db}\rho^{2,0}=\lambda\rho^{2,0}
\end{align*}
\end{prop}
\begin{proof}
If $\lambda=0$, then we have $\rho=\rho^{2,0}+\overline{\rho^{2,0}}$.
We obtain that $\rho^{2,0}$ is closed for $\rho$ is closed.
Thus $\rho^{2,0}$ is $\db$-harmonic.

If $\lambda\neq0$, then we have
\begin{align}\label{eq-relationshipHS}
0=\dd\rho^{1,1}+\db\rho^{2,0}=\lambda\dd\omega+\db\rho^{2,0}
\end{align}
and $\dd\rho^{2,0}=0$.
On the other hand, we have (see Section \ref{subsec-Bismut})
\begin{align*}
\rho^{2,0}=-\dd\db^*\omega,
\end{align*}
which implies
\begin{align*}
\rho^{2,0}
=\dd*\dd*\omega=\dd*\dd\omega
=-\frac{1}{\lambda}\dd*\db\rho^{2,0}
=-\frac{1}{\lambda}*\dd*\db\rho^{2,0}
=\frac{1}{\lambda}\db^*\db\rho^{2,0}
=\frac{1}{\lambda}\D_{\db}\rho^{2,0}
\end{align*}
The second equality is because of $*\omega=\omega$ in dimension 2.
The third equality uses equation \eqref{eq-relationshipHS}.
And the fourth equality uses the fact that $*\eta=\eta$ for arbitrary $(2,0)$-form in dimension 2, which is an application of Lefschetz theorem to primitive form (see e.g., \cite{Wells}).
\end{proof}

\begin{proof}[Another proof of Theorem \ref{the-introMain} in the case of surfaces.]

For convenience, we denote $\varphi=\rho^{2,0}$.

{\bf Case I $\lambda<0$.}
Using Proposition \ref{prop-observation}, we have
\begin{align*}
(\varphi,\varphi)_2=(\lambda\Delta_{\db}\varphi,\varphi)_2
=\lambda(\db^*\db\varphi,\varphi)_2
=\lambda(\db\varphi,\db\varphi)_2.
\end{align*}
where $(\cdot,\cdot)_2$ is the $L^2$ Hermitian inner product.
So we obtain $\varphi=0$ for $\lambda$ is a negative number.
Notice 
$\lambda\dd\omega=\dd\rho^{1,1}=-\db\varphi=0$, which means that $\omega$ is K{\"a}hler.

{\bf Case II $\lambda=0$.}
From Proposition \ref{prop-observation} we know that $\varphi$ is a $\db$-harmonic $(2,0)$-form. 
Theorem \ref{thm-introVan2} tell us that $\omega$ is K{\"a}hler or $\varphi=0$.
In fact, in both cases we have $\varphi=0$ because for K{\"a}hler metrics, the Bismut Ricci form is precisely the K{\"a}hler Ricci form which only has $(1,1)$-part.

{\bf Case III $\lambda>0$.}
Applying Corollary \ref{cor-BochnerPLC} to $\varphi$, we get
\begin{equation}\label{eq-BochnerMain}
\begin{split}
\D|\varphi|^2
&=(\D \varphi,\varphi)+(\varphi,\D \varphi)
			+|\n \varphi|^2+|\nb \varphi|^2
			-(\varphi,P\circ \varphi)
			-(\varphi,Q\circ \varphi)
\\
&=(\D \varphi,\varphi)+(\varphi,\D \varphi)
			+|\n \varphi|^2+|\nb \varphi|^2
			+(2\lambda+|T|^2)|\varphi|^2
\end{split}
\end{equation}
The second row is because of Proposition \ref{prop-observation} and Lemma \ref{lem-Htwo2}.

Using Proposition \ref{prop-LapForm}, we get
\begin{align*}
(\Delta\varphi,\varphi)_2
=-(\Delta_{\db}\varphi,\varphi)_2
			+\sqrt{-1}(L^{*}_{\db^*\omega}\db\varphi,\varphi)_2
\end{align*}
By the definition of $L^{*}_{(\cdot)}(\cdot)$ (see Section \ref{sec-Lfs}),
\begin{align*}
(L^{*}_{\db^*\omega}\db\varphi,\varphi)_2
&=\int_{M} (L^{*}_{\db^*\omega}\db\varphi,\varphi)dV
=\int_{M} (\db\varphi,\dd^{*}\omega\wedge\varphi)dV
=(\db\varphi,\dd^{*}\omega\wedge\varphi)_2
\end{align*}
By direct computation, we obtain
\begin{align*}
(\db\varphi,\dd^{*}\omega\wedge\varphi)_2
&=(\db^{*}\omega\wedge\bar{\varphi},\dd\bar{\varphi})_2
=\int_{M}\db^{*}\omega\wedge\bar{\varphi}\wedge *\db\varphi
\\
&=-\lambda\int_{M}\db^{*}\omega\wedge\bar{\varphi}\wedge *\dd\omega
=-\lambda\int_{M}\db^{*}\omega\wedge\bar{\varphi}\wedge *\dd*\omega
\\
&=\lambda\int_{M}\db^{*}\omega\wedge\bar{\varphi}\wedge \db^*\omega
=\lambda\int_{M}\db^{*}\omega\wedge \db^*\omega\wedge\bar{\varphi}
\\
&=0
\end{align*}
The second row is because of
$\lambda\dd\omega=\dd\rho^{1,1}=-\db\varphi$.
And the last row uses the fact that $\db^*\omega\wedge\db^*\omega=0$, which holds because $\db^*\omega$ is an odd degree form.
So we obtain
\begin{align*}
(\Delta\varphi,\varphi)_2
=-(\Delta_{\db}\varphi,\varphi)_2=-\lambda\Vert\varphi\Vert^2.
\end{align*}
Similarly, we have
\begin{align*}
(\varphi,\Delta\varphi)_2=-\lambda\Vert\varphi\Vert^2.
\end{align*}
Integrating equation \eqref{eq-BochnerMain} on the manifold and noticing Proposition \ref{prop-GauduchonIntLap=0}, we have
\begin{align*}
0
&=\Vert\n \varphi\Vert^2+\Vert\nb \varphi\Vert^2
			+\int_{M}|T|^2|\varphi|^2dV
\end{align*}
which implies $\nb\varphi=0$.
Since
$
\lambda\dd\omega=\dd\rho^{1,1}=-\db\varphi=-\nb\varphi=0,
$
the metric $\omega$ is K{\"a}hler and $\varphi=0$.
\end{proof}

\section{Appendix}
In this appendix, we give the proof of some lemmas used earlier.

\subsection{Proof of Lemma \ref{lem-exOf3derivative}}

\begin{proof}
Firstly, we recall that for a connection $D$ with torsion $H$ and curvature tensor $\R$, we have
\begin{align}\label{eq-ex2DwithT}
D^{2}_{xy}A-D^{2}_{yx}A=\R_{xy}\circ A-D_{H(x,y)}A
\end{align}
Applying it to Chern connection, we get
\begin{align*}
\n^{3}_{k\bar{k}\bar{l}}A
&=\n_{k}(\n^{2}A)_{\bar{k}\bar{l}}
=\n_{k}(\n^{2}_{\bar{k}\bar{l}}A)
\\
&=\n_{k}(\n^{2}_{\bar{l}\bar{k}}A
	-T_{\bar{k}\bar{l}}^{\bar{s}}\n_{\bar{s}}A)
\\
&=\n^{3}_{k\bar{l}\bar{k}}A
	-\n_{k}T_{\bar{k}\bar{l}}^{\bar{s}}\n_{\bar{s}}A
	-T_{\bar{k}\bar{l}}^{\bar{s}}\n_{k}\n_{\bar{s}}A
\end{align*}
The second row is because of $\Omega_{\bar{k}\bar{l}\cdots}=0$.

Similarly, applying formula \eqref{eq-ex2DwithT} to $\nb A$, we get
\begin{align*}
\n^{3}_{k\bar{l}\bar{k}}A
=\n^{2}_{k\bar{l}}(\nb A)_{\bar{k}}
=\n^{3}_{\bar{l}k\bar{k}}A
	+\Omega_{k\bar{l}}\circ\n_{\bar{k}}A
\end{align*}
Combining above together, we obtain
\begin{align*}
\n^{3}_{k\bar{k}\bar{l}}A-\n^{3}_{\bar{l}k\bar{k}}A
=\Omega_{k\bar{l}}\circ\n_{\bar{k}}A
	-\n_{k}T_{\bar{k}\bar{l}}^{\bar{s}}\n_{\bar{s}}A
	-T_{\bar{k}\bar{l}}^{\bar{s}}\n_{k}\n_{\bar{s}}A
\end{align*}
And we complete the proof.
\end{proof}

\subsection{Proof of Lemma \ref{lem-compDb}}
\begin{proof}
Since the test form $\beta$ can be chosen that is supported in a coordinates neighborhood, we can do calculation locally.
Assume
\begin{align*}
\beta=\beta_{a_1\cdots a_p \bar{b}_1\cdots \bar{b}_q}
				dz^{a_1}\wedge\cdots\wedge dz^{a_p}
				\wedge
				d\bar{z}^{b_1}\wedge\cdots\wedge d\bar{z}^{b_q}.
\end{align*}
We have
\begin{align*}
(\db^*\alpha,\beta)_2
&=\int_{M} (\db^*\alpha,\beta) \det g\\
&=\int_{M} \eta_{i_1\cdots i_p \bar{j}_1\cdots \bar{j}_q}
				\overline{\beta_{a_1\cdots a_p \bar{b}_1\cdots \bar{b}_q}}
				g^{\bar{a}_1 i_1}\cdots g^{\bar{a}_p i_p}
				g^{\bar{j}_1 b_1}\cdots g^{\bar{j}_q b_q}
				\det g		
.		
\end{align*}
Here we use $(\cdot,\cdot)_2$ and $(\cdot,\cdot)$ to denote $L^2$ Hermitian inner product and pointwise Hermitian inner product, respectively.
By direct computation, we get
\begin{align*}
\db\beta=(-1)^{p+q}\dd_{\bar{l}}
				\beta_{a_1\cdots a_p \bar{b}_1\cdots \bar{b}_q}
				dz^{a_1}\wedge\cdots\wedge dz^{a_p}
				\wedge
				d\bar{z}^{b_1}\wedge\cdots\wedge d\bar{z}^{b_q}
				\wedge
				d\bar{z}^l
\end{align*}
Then
\begin{align*}
(\db^*\alpha,\beta)_2
&=(\alpha,\db\beta)_2
=\int_{M} (\alpha,\db\beta)\det g\\
&=(-1)^{p+q}\int_{M} 
			\alpha_{i_1\cdots i_p \bar{j}_1\cdots \bar{j}_q \bar{k}}
			\dd_l \overline{\beta_{a_1\cdots a_p \bar{b}_1\cdots \bar{b}_q}}
			g^{\bar{a}_1 i_1}\cdots g^{\bar{a}_p i_p}
			g^{\bar{j}_1 b_1}\cdots g^{\bar{j}_q b_q}
			g^{\bar{k}l}
			\det g
\\
&=(-1)^{p+q+1}\int_{M} 
			\overline{\beta_{a_1\cdots a_p \bar{b}_1\cdots \bar{b}_q}}
			\dd_l
		\Big(		
			\alpha_{i_1\cdots i_p \bar{j}_1\cdots \bar{j}_q \bar{k}}
			g^{\bar{a}_1 i_1}\cdots g^{\bar{a}_p i_p}
			g^{\bar{j}_1 b_1}\cdots g^{\bar{j}_q b_q}
			g^{\bar{k}l}
			\det g
		\Big).
\end{align*}
The arbitrariness of $\beta$ implies
\begin{align*}
\eta_{i_1\cdots i_p \bar{j}_1\cdots \bar{j}_q}
				g^{\bar{a}_1 i_1}\cdots g^{\bar{a}_p i_p}
				g^{\bar{j}_1 b_1}\cdots g^{\bar{j}_q b_q}
				\det g
&=(-1)^{p+q+1}
\dd_l
		\Big(		
			\alpha_{i_1\cdots i_p \bar{j}_1\cdots \bar{j}_q \bar{k}}
			g^{\bar{a}_1 i_1}\cdots g^{\bar{a}_p i_p}
\\
&\qquad\qquad
			g^{\bar{j}_1 b_1}\cdots g^{\bar{j}_q b_q}
			g^{\bar{k}l}
			\det g
		\Big)
\end{align*}
which gives
\begin{align*}
(-1)^{p+q+1}\eta_{c_1\cdots c_p \bar{d}_1\cdots \bar{d}_q}
&=g^{\bar{k}l}\dd_l 
						\alpha_{c_1\cdots c_p \bar{d}_1\cdots \bar{d}_q \bar{k}}
			+g^{\bar{k}l}\sum_{m=1}^{p}
						\alpha_{c_1\cdots i_m \cdots c_p \bar{d}_1\cdots \bar{d}_q \bar{k}}
						g_{c_m \bar{a}_m}\dd_l g^{\bar{a}_m i_m} 
\\
	&\qquad+
						g^{\bar{k}l}\sum_{m=1}^{q}
			\alpha_{c_1\cdots c_p\bar{d}_1\cdots\bar{j}_m\cdots \bar{d}_q\bar{k}}
						g_{b_m \bar{d}_m}\dd_l g^{\bar{j}_m b_m}
\\
	&\qquad+
			\alpha_{c_1\cdots c_p \bar{d}_1\cdots \bar{d}_q \bar{k}}
					\dd_{l}g^{\bar{k}l}
			+g^{\bar{k}l}\alpha_{c_1\cdots c_p \bar{d}_1\cdots \bar{d}_q \bar{k}}
					g^{\bar{t}s}\dd_{l}g_{s\bar{t}}
\\
\end{align*}
Recall the derivative formula of inverse matrix
\begin{align*}
\dd_{l}g^{\bar{j}i}=-g^{\bar{j}p}g^{\bar{q}i}\dd_{l}g_{p\bar{q}}
=-g^{\bar{j}p}\Gamma^{i}_{lp}
\end{align*}
Changing some subscripts, we get
\begin{align*}
(-1)^{p+q+1}\eta_{c_1\cdots c_p \bar{d}_1\cdots \bar{d}_q}
&=g^{\bar{k}l}\dd_l 
						\alpha_{c_1\cdots c_p \bar{d}_1\cdots \bar{d}_q \bar{k}}
			-g^{\bar{k}l}\sum_{m=1}^{p}
						\alpha_{c_1\cdots s \cdots c_p \bar{d}_1\cdots \bar{d}_q \bar{k}}
						\Gamma^{s}_{lc_m} 
\\
	&\qquad-
						g^{\bar{k}l}g^{\bar{t}s}\sum_{m=1}^{q}
			\alpha_{c_1\cdots c_p\bar{d}_1\cdots\bar{t}\cdots \bar{d}_q\bar{k}}				
						\dd_l g_{s\bar{d}_m}
\\
	&\qquad-
			\alpha_{c_1\cdots c_p \bar{d}_1\cdots \bar{d}_q \bar{k}}
					g^{\bar{k}l}g^{\bar{t}s}\dd_{s}g_{l\bar{t}}
			+\alpha_{c_1\cdots c_p \bar{d}_1\cdots \bar{d}_q \bar{k}}
					g^{\bar{k}l}g^{\bar{t}s}\dd_{l}g_{s\bar{t}}
\\
\end{align*}
Notice that
\begin{align*}
-g^{\bar{k}l}g^{\bar{t}s}\sum_{m=1}^{q}
			\alpha_{c_1\cdots c_p\bar{d}_1\cdots\bar{t}\cdots \bar{d}_q\bar{k}}				
			\dd_l g_{s\bar{d}_m}
&=-\frac{1}{2}g^{\bar{k}l}g^{\bar{t}s}\sum_{m=1}^{q}
			\alpha_{c_1\cdots c_p\bar{d}_1\cdots\bar{t}\cdots \bar{d}_q\bar{k}}
			(\dd_l g_{s\bar{d}_m}-\dd_s g_{l\bar{d}_m})
\\
&=\frac{1}{2}g^{\bar{k}l}g^{\bar{t}s}\sum_{m=1}^{q}
			\alpha_{c_1\cdots c_p\bar{d}_1\cdots\bar{t}\cdots \bar{d}_q\bar{k}}
			T_{sl\bar{d}_m}
\end{align*}
Then we obtain
\begin{align*}
(-1)^{p+q+1}\eta_{c_1\cdots c_p \bar{d}_1\cdots \bar{d}_q}
&=g^{\bar{k}l}\nabla_l 
						\alpha_{c_1\cdots c_p \bar{d}_1\cdots \bar{d}_q \bar{k}} 
		+\frac{1}{2}g^{\bar{k}l}g^{\bar{t}s}\sum_{m=1}^{q}
					\alpha_{c_1\cdots c_p\bar{d}_1\cdots\bar{t}\cdots \bar{d}_q\bar{k}}
					T_{sl\bar{d}_m}
\\
	&\qquad+
			g^{\bar{k}l}\alpha_{c_1\cdots c_p \bar{d}_1\cdots \bar{d}_q \bar{k}}
			T_{ls}^{s}
\end{align*}
\end{proof}

\bibliographystyle{plain}
\bibliography{beccs}

\end{document}